\documentclass[11pt]{amsart}
\usepackage{amsmath,amstext,amsbsy,amssymb,amscd,enumerate}
\usepackage{amsxtra}
\usepackage{amscd}
\usepackage{amsthm}
\usepackage{amsfonts}
\usepackage{graphicx}
\usepackage{amssymb}
\usepackage{eucal}
\usepackage{color}
\usepackage{multirow}
\usepackage[enableskew,vcentermath]{youngtab}

\newtheorem{theorem}{Theorem}[section]

\newtheorem{lemma}[theorem]{Lemma}
\newtheorem{proposition}[theorem]{Proposition}
\newtheorem{corollary}[theorem]{Corollary}
\newtheorem{example}[theorem]{Example}
\theoremstyle{definition}
\newtheorem{definition}[theorem]{Definition}

\theoremstyle{remark}
\newtheorem{remark}[theorem]{Remark}

\definecolor{A}{rgb}{.75,1,.75}

\numberwithin{equation}{section}
\numberwithin{equation}{section}

\newcommand{\R}{\mathbb R}
\newcommand{\C}{\mathbb C}
\newcommand{\Z}{\mathbb Z}

\newcommand{\g}{{\mathfrak{g}}}
\newcommand{\gl}{{\mathfrak{gl}}}
\newcommand{\glmn}{\mathfrak{gl}(m|n)}
\newcommand{\lv}{\mathfrak{l}}
\newcommand{\h}{\mathfrak{h}}
\newcommand{\uu}{\mathfrak{u}}

\newcommand{\mc}{\mathcal}
\newcommand{\mf}{\mathfrak}
\newcommand{\ind}{{\rm Ind} }
\newcommand{\res}{{\rm Res} }
\newcommand{\la}{\lambda}

\newcommand{\n}{\mf{n}}

\newcommand{\OO}{\mc O}
{\vskip-\lastskip\medskip
  \noindent
  {\em #1.}\enspace
  }%
{\qed\par\medskip
  }

\begin{document}

\title[]{Equivalence of blocks for the general linear Lie superalgebra}

\author[Cheng]{Shun-Jen Cheng}
\address{Institute of Mathematics, Academia Sinica, Taipei,
Taiwan 10617} \email{chengsj@math.sinica.edu.tw}

\author[Mazorchuk]{Volodymyr Mazorchuk}
\address{Department of Mathematics, Uppsala University, Box 480, SE-751 06, Uppsala, Sweden}
\email{mazor@math.uu.se}

\author[Wang]{Weiqiang Wang}
\address{Department of Mathematics, University of Virginia,
Charlottesville,VA 22904, USA} \email{ww9c@virginia.edu}

\begin{abstract}
We develop a reduction procedure which provides an equivalence (as highest weight
categories) from
an arbitrary block (defined in terms of the central character and the integral Weyl
group) of the BGG category $\OO$ for a
general linear Lie superalgebra to an integral block of $\OO$ for (possibly a direct sum
of)
general linear Lie superalgebras. We also establish indecomposability of blocks of
$\OO$.
\end{abstract}



\maketitle

\let\thefootnote\relax


\section{Introduction}

Recently, Brundan's Kazhdan-Lusztig-type conjecture \cite{Br} on
irreducible and tilting characters of the general linear Lie
superalgebra $\glmn$ in terms of canonical bases of Fock spaces has been
established by the first and third authors jointly with Ngau Lam
\cite{CLW}. In these papers only the integral part of category $\OO$
for $\glmn$ was considered. A natural next step is to understand
the structure of a block for an arbitrary weight and ask for a
possible reduction via equivalence to an integral weight block.
In the setting of semisimple Lie algebras, this was achieved by
Soergel \cite{Soe} by establishing an equivalence between a general
block and an integral block associated to the so-called integral Weyl group.
However, the approach of Soergel is not directly applicable to Lie superalgebras,
since blocks of $\OO$ for Lie superalgebras are not controlled by Coxeter groups in general.

There is an alternative approach in the case of Lie algebras developed by
Mathieu \cite{Ma} and Kashiwara-Tanisaki \cite{KT}. It uses the localization of
the universal enveloping algebra with respect to certain root elements and polynomial
extensions of conjugation automorphisms for this localization. In this note we
develop yet another approach to this problem based on twisting functors from \cite{Ar}.
The first step in the definition of
twisting functors uses, similarly to the approach of Mathieu and Kashiwara-Tanisaki, the localization of
the universal enveloping algebra with respect to certain root elements. The difference is in the
second step where, instead of polynomial extensions of conjugation automorphisms for the localization,
one uses Arkhipov's twisting trick from \cite{Ar}. These twisting functors admit straightforward
generalizations to Lie superalgebras and so allow us to establish, for Lie
algebras and superalgebras of type $A$, an explicit equivalence (as highest weight categories)
between a block of $\OO$ for an arbitrary weight and an integral block of $\OO$ for a certain
subalgebra, where the equivalence is built upon parabolic induction and twisting functors.

Let us explain our approach in more detail. Let $\g$ be a Lie algebra of ADE type. For a weight $\la$, consider the corresponding integral subalgebra $\g_{[\la]}$ of $\g$ (see \eqref{defn:g:lambda}). If $\g_{[\la]}$ is a Levi subalgebra, then the usual parabolic induction functor gives the desired equivalence of blocks, with the block for (the semisimple part of) $\g_{[\la]}$ being integral. Next we show that a twisting functor associated to a simple root $\alpha$ such that $\langle \la, \alpha^\vee \rangle \not \in \Z$ is always an equivalence (as highest weight categories) of blocks. It remains to observe that for a type A Lie algebra it is always possible to find a sequence of such twisting functors that connects an arbitrary block with a block for which the integral subalgebra is Levi.

Now we are ready to generalize the above approach to the Lie
superalgebra $\g =\glmn$. One ``super phenomenon'' is the existence of
non-conjugate Borel subalgebras of $\g$. Given two Borel subalgebras
$\mf b$ and $\mf b'$ that are related by an odd reflection with
respect to an isotropic odd simple root, it is known that the
categories $\OO$ relative to $\mf b$ and $\mf b'$ coincide, but the identity functor
does not preserve the structure of highest weight categories in general
\cite[Section~6]{CLW}. This statement is easily seen to be valid for
the version of category $\OO$ with arbitrary (i.e., non-integral)
weights. We show here that the block equivalence with respect
to an odd reflection associated to an isotropic odd simple root $\alpha$
is indeed an equivalence of highest weight categories
under the assumption that $\langle \la, \alpha^\vee \rangle \not \in \Z$. In the
case of $\glmn$, we show that it is possible to use a sequence of
twisting functors and such odd reflections to reduce to the setting of an
integral subalgebra of Levi type. The same type of parabolic induction functor
as in the case of semisimple Lie algebras completes the reduction
process. Finally, we show that a block as defined via its central character
and integral Weyl group is indeed indecomposable just as for semisimple Lie algebras.

The above strategy works for a large set of, but not all,
non-integral weights for other types of simple and Kac-Moody Lie
(super)algebras; compare with Fiebig \cite{Fie}.

This note is organized as follows. In Section~\ref{sec:LA}, we develop our approach in the Lie algebra setting.
The block equivalences via parabolic induction and twisting functors are formulated and established.
In Section~\ref{sec:superLA}, we generalize it to the Lie superalgebra setting.
For basics on category $\OO$ we refer to the book of Humphreys \cite{Hu}, and for basics on
Lie superalgebras we refer to the new book \cite{CWbook}.

\vspace{.2cm}
{\bf Acknowledgement.} The first author is partially supported by an
NSC grant of the R.O.C. The second author is partially supported by
the Royal Swedish Academy of Sciences and the Swedish Research Council.
The third author is partially supported by NSF
DMS-1101268, and he thanks T.~Tanisaki for some helpful
correspondence. We thank Kevin Coulembier for pointing out an oversight
in the original formulation of Theorem~\ref{th:twist}.

\section{Equivalence of blocks for Lie algebras}

\label{sec:LA}

\subsection{The setup}\label{sec:setup}

Let $\g$ be a complex semisimple or reductive Lie algebra with a
Cartan subalgebra $\h$, root system $\Phi$, and Weyl group $W$.
We denote by $U(\g)$ the universal enveloping algebra of $\g$
and by $Z(\g)$ the center of $U(\g)$. Let $\Pi$ be a simple system
of $\Phi$ and denote by $\Phi^+$ the corresponding set of positive
roots. For $\alpha\in\Phi$, we let $\alpha^\vee$ stand for the
corresponding coroot. Denote by $\rho$ the half of the sum of all positive roots,
and consider the usual dot- (i.e., $\rho$-shifted) action of the Weyl group
given by $w \cdot \la
=w(\la+\rho) -\rho$, for $w\in W$ and $\la \in \h^*$. We have the
corresponding triangular decomposition $\g =\n^- \oplus \h \oplus
\n$, where $\mf{n}=\bigoplus_{\alpha\in\Phi^+}\g_\alpha$ and
$\mf{n}^-=\bigoplus_{\alpha\in-\Phi^+}\g_\alpha$. Let ${}^{\g}\OO$
be the Bernstein-Gelfand-Gelfand (BGG) category of $\g$-modules
(see \cite{BGG} or \cite[Chapter 1]{Hu}). We shall drop the superscript
$\g$ and write simply $\OO$, when $\g$ is clear from the context.

For each $\la \in \h^*$ we have the integral root system
$\Phi_{\la} =\{\alpha \in \Phi \mid \langle \la, \alpha^\vee \rangle
\in \Z\}$ associated to $\la$ and the corresponding
integral Weyl group $W_{\la} $, which is the subgroup of $W$ generated by
all reflections $s_\alpha$, $\alpha \in \Phi_{\la}$. The weight $\la$ is
called integral if $W=W_{\la}$.
When $\g$ is of ADE type, we define
\begin{align}\label{defn:g:lambda}
\g_{[\la]} =\h \oplus \bigoplus_{\alpha \in \Phi_{\la}} \g_\alpha,
\end{align}
which is a Lie subalgebra of $\g$, called the {\em integral
subalgebra} associated to $\la$. If $\g$ is not of ADE type, then
$\g_{[\la]}$ may not be closed under the Lie bracket.

For $\la\in\h^*$ we denote by $M(\lambda)$ the Verma module with highest weight
$\lambda$, by $L(\lambda)$ the unique simple quotient of $M(\lambda)$
and by $\chi_\la:Z(\g)\to\mathbb{C}$ the central character of $M(\lambda)$ (see \cite[Chapter~1]{Hu}).
Denote by $\OO_{\chi_\la}$ the Serre subcategory of $\OO$ generated by the simple objects
$L(w\cdot \la)$, $w\in W$ (these are exactly the simple objects of $\OO$ with central character $\chi_\la$,
see \cite[Section~1.10]{Hu}).
Then the action of $Z(\g)$ gives rise to a decomposition of $\mc O$
as follows: $\mc O=\bigoplus_{\la\in\h^*/(W,\cdot)}\mc O_{\chi_\la}$.
If $\la$ is integral, then $\mc O_{\chi_\la}$ is indecomposable (see \cite[Section~1.13]{Hu}). However,
$\mc O_{\chi_\la}$ is decomposable in general.

Denote by $\OO_{\la}$ the Serre subcategory of $\OO$ generated by
$L(w\cdot \la)$, $w\in W_{\la}$. Define an
equivalence relation $\sim$ on $\h^*$ by declaring $\la\sim\mu$ if $\mu\in
W_{\la}\cdot\la$. Let $\h^*_{\mathrm{dom}}$ denote the set of all dominant weights
with respect to the dot-action, that is the set of all $\lambda$ such that $\langle\lambda+\rho,\alpha^{\vee}\rangle
\not\in\mathbb{Z}^{<0}$ for all $\alpha\in\Phi^+$ (see \cite[Section~3.5]{Hu}). Each such dominant
$\lambda$ is the maximum element in $W_{\la}\cdot\la$ with respect to the natural order $\leq$ on $\h^*$
(given by $\mu\leq \nu$ if and only if $\nu-\mu\in \mathbb{Z}^{\geq 0}\Phi^+$). Then
we have a refined decomposition of $\mc O$
into {\em indecomposable} blocks (see \cite[Theorem 4.9]{Hu}): $\mc O=\oplus_{\la\in\h^*_{\mathrm{dom}}}\mc O_{\la}$.

\subsection{Twisting functors}\label{sec2.2}

For $\alpha\in\Pi$ fix a nonzero $X\in\mathfrak{g}_{-\alpha}$ and let $U'_{\alpha}$ be the
(Ore) localization of $U(\g)$ with respect to powers of $X$. Then $U(\g)$ is a subalgebra of the
associative algebra $U'_{\alpha}$ and the quotient $U_{\alpha}:=U'_{\alpha}/U(\g)$ has the
induced structure of a $U(\g)\text{-}U(\g)$--bimodule. Let $\varphi=\varphi_{\alpha}$ be an
automorphism of $\mathfrak{g}$ that maps $\mathfrak{g}_{\beta}$ to $\mathfrak{g}_{s_{\alpha}(\beta)}$
for all $\beta\in\Pi$, where $s_{\alpha}$ is the simple reflection corresponding to $\alpha$.
Finally, consider the bimodule ${}^{\varphi}U_{\alpha}$, which is obtained from
$U_{\alpha}$ by twisting the left action of $U(\g)$ by $\varphi$
(i.e. $a\cdot u\cdot b:=\varphi(a)ub$ for all $a,b\in U(\g)$ and $u\in U_{\alpha}$). Tensoring with
${}^{\varphi}U_{\alpha}$ defines an endofunctor $T_{\alpha}$ of $\mc O$,
called the {\em twisting functor}. This functor was originally defined
by Arkhipov in \cite{Ar} and further investigated in more detail in \cite{HAS,KM}.

Let $\mathcal{D}^b(\mc O)$ denote the bounded derived category of $\mc O$ and let
$\mathcal{L}T_{\alpha}:\mathcal{D}^b(\mc O)\to \mathcal{D}^b(\mc O)$ be the corresponding
left derived endofunctor. Furthermore, for $i\geq 0$, let $\mathcal{L}_iT_{\alpha}:\mc O\to\mc O$
denote the $i$-th cohomology of $\mathcal{L}T_{\alpha}$.
Let us recall some basic properties of $T_{\alpha}$ (see, e.g.,~\cite{HAS}):
\begin{enumerate}[$($I$)$]
\item\label{twist1} $T_{\alpha}$ is right exact;
\item\label{twist2} $T_{\alpha}$ preserves $\OO_{\chi_\la}$, for each $\la \in \h^*$;
\item\label{twist3} $\mathcal{L}_iT_{\alpha}=0$ for $i>1$;
\item\label{twist4} $\mathcal{L}_1T_{\alpha}$ is isomorphic to the
functor of taking the maximal submodule on which the action of
$\mathfrak{g}_{-\alpha}$ is locally nilpotent.
\end{enumerate}

\subsection{Equivalences using twisting functors}

Let $\la \in \h^*$ and $\alpha\in\Pi$ be such that
$\langle \la, \alpha^\vee \rangle \not\in \Z$. Denote by $\OO_{\chi_\la}^{(\alpha)}$ the Serre subcategory of
$\OO_{\chi_\la}$ generated by all $L(w\cdot \la)$, $w\in W$, such that 
$\langle w\cdot \la, \alpha^\vee \rangle \not\in \Z$ (the set of all $w$ having the latter property will
be denoted $W(\lambda,\alpha)$). Then $\OO_{\chi_\la}^{(\alpha)}$ is a direct summand of 
$\OO_{\chi_\la}$.

\begin{theorem} \label{th:twist}
Let $\la \in \h^*$ and $\alpha\in\Pi$ be such that
$\langle \la, \alpha^\vee \rangle \not\in \Z$.
\begin{enumerate}[$($a$)$]
\item\label{th:twist.1}
The functor $T_\alpha: \OO_{\chi_\la}^{(\alpha)} \longrightarrow \OO_{\chi_\la}^{(\alpha)}$ is
an autoequivalence (of highest weight categories)
sending $M(w\cdot \la)$ to $M(s_\alpha w\cdot \la)$ and $L(w\cdot \la)$ to $L(s_\alpha w\cdot \la)$
for all $w\in W(\lambda,\alpha)$.
\item\label{th:twist.2}
For any $\mu\in W(\lambda,\alpha)\cdot \la$, the functor $T_\alpha$ induces an equivalence between
$\OO_{\mu}$ and $\OO_{s_{\alpha}\cdot\mu}$.
\end{enumerate}
\end{theorem}

\begin{proof}
Since $\langle \la, \alpha^\vee \rangle \not\in \Z$, the classical theory of highest weight modules over
$\mathfrak{sl}(2)$ (see, e.g., \cite[Chapter~3]{Maz}) implies that the action
of $\mathfrak{g}_{-\alpha}$ is injectively on each module from $\OO_{\chi_\la}^{(\alpha)}$.
Therefore Properties \eqref{twist3} and \eqref{twist4} imply that 
$T_\alpha: \OO_{\chi_\la}^{(\alpha)} \rightarrow \OO_{\chi_\la}^{(\alpha)}$ is exact.
The functor $T_{\alpha}$ maps Verma modules to Verma
modules (this follows from, e.g., \cite[Lemma~3]{KM}).

Let us now check that $T_\alpha$ sends $L(w\cdot \la)$ to $L(s_\alpha w\cdot \la)$.
Set $M:=T_\alpha\,L(w\cdot \la)$. We claim that $s_\alpha w\cdot \la$ is a highest weight
of $M$. Indeed, if we localize $L(w\cdot \la)$ and then factor out the copy of $L(w\cdot \la)$ in
the localization, the resulting module will have a one-dimensional weight space of weight
$w\cdot \la+\alpha$. The twisting procedure maps $w\cdot \la+\alpha$
to $s_\alpha w\cdot \la$ and thus $s_\alpha w\cdot \la\in\mathrm{supp}(M)$. By construction,
$(s_\alpha w\cdot \la)+\alpha\not\in \mathrm{supp}(M)$. Since $\varphi_{\alpha}$ leaves the set of
all positive roots different from $\alpha$ invariant, it follows that
$(s_\alpha w\cdot \la)+\beta\not\in \mathrm{supp}(M)$ for any $\beta\in\Phi^+\setminus\{\alpha\}$.
Hence $s_\alpha w\cdot \la$ is a highest weight of $M$. In particular,
$M\neq 0$ (i.e., $T_\alpha$ does not annihilate any modules).

Denote by $\mathfrak{a}$ the $\mathfrak{sl}(2)$-subalgebra of $\mathfrak{g}$ generated by
$\mathfrak{g}_{\pm\alpha}$. Starting with $N\in\{L(w\cdot \la),M\}$ we can do two things: first apply $T_\alpha$
to $N$ and then consider the result as an $\mathfrak{a}$-module, or first consider $N$ and as
an $\mathfrak{a}$-module and then apply the $\mathfrak{a}$-version of $T_\alpha$ to it.
By \cite[Lemma~3]{KM} (note that the integrality assumption in that
lemma just reflects the general setup of \cite{KM} and is not used in the proof), both these
ways produce isomorphic $\mathfrak{a}$-modules. In the $\mathfrak{sl}(2)$-case it is straightforward
to verify (cf.~\cite[Chapter~5]{Maz}) that the characters of $T_\alpha\, M$ and $L(w\cdot \la)$ coincide.
It follows that the characters of $T_\alpha\, M$ and $L(w\cdot \la)$ coincide in the genral case
and hence $T_\alpha\, M\cong L(w\cdot \la)$, as any simple module in $\OO$ is uniquely determined by
its character.

Now, by the above we know that $L(s_\alpha w\cdot \la)$ is a simple subquotient of $M$. We claim
that $M$ is simple. Indeed, assume that this were not the case. Then $T_\alpha\, M$ cannot be simple,
since $T_\alpha$ is exact and does not annihilate any modules. This contradicts the conclusion in the
previous paragraph.

Thus, we know that $T_\alpha$ is an exact endofunctor of $\OO_{\chi_\la}^{(\alpha)}$ which sends simple
modules to simple modules. Let $K_\alpha$ denote the right adjoint of $T_\alpha$.
By \cite[Theorem~4.1]{HAS}, we have $K_\alpha\cong \mathbf{d}T_\alpha\mathbf{d}$, where
$\mathbf{d}:\OO_{\chi_\la}^{(\alpha)}\to\OO_{\chi_\la}^{(\alpha)}$ 
is the usual duality (i.e., a contravariant autoequivalence which
preserves isoclasses of simple modules, see \cite[3.2]{Hu}). Since $\mathbf{d}$ is exact, it follows
that $K_\alpha$ is exact as well. Furthermore, using the above we compute:
\begin{displaymath}
K_\alpha\, L(w\cdot \la)\cong
\mathbf{d}T_\alpha\mathbf{d}\, L(w\cdot \la)\cong
\mathbf{d}T_\alpha\, L(w\cdot \la)\cong
 \mathbf{d}\,L(s_\alpha w\cdot \la)\cong L(s_\alpha w\cdot \la),
\end{displaymath}
which means that $K_\alpha$ also sends simple modules to simple modules.
Now the fact that $T_\alpha$ and $K_\alpha$ are mutually inverse equivalences of categories
follows by standard arguments using induction on the length of a module and the Short Five Lemma.
This proves \eqref{th:twist.1}, and \eqref{th:twist.2} follows from
\eqref{th:twist.1} by restriction.
\end{proof}

\begin{example}
If $\g=\gl(2)$, then every non-integral block of $\mc O$ is semisimple. For each non-integral
central characters there exist exactly two non-isomorphic simple highest weight modules
(see e.g. \cite[Section~3.2]{Maz}) and $T_\alpha$ swaps them.
\end{example}

\subsection{Parabolic induction from a Levi subalgebra}

For a weight module $M$ and a weight $\mu$ we denote by $M_{\mu}$ the $\mu$-weight space in $M$.
For $\Pi'\subseteq \Pi$ consider the corresponding Levi subalgebra $\mf{l}=\mf{l}(\Pi')$
(which, by definition, contains $\mathfrak{h}$).
Let $\uu$ be the nilpotent radical of $\mf{l}+\mathfrak{b}$, where $\mathfrak{b}$ is the Borel subalgebra
corresponding to $\Phi^+$. For $M\in\OO$ the maximal subspace $M^{\uu}$ of $M$
on which $\uu$ acts trivially inherits the natural structure of an $\mf{l}$-module by restriction.
The following proposition is standard (see e.g.~\cite[Lemma~2]{J}).

\begin{proposition} \label{prop:ind}
Let $\g$ be a semisimple Lie algebra, $\la \in \h^*$, and suppose
that $\g_{[\la]}=\mf{l}$ is a Levi subalgebra of $\g$. Then the parabolic induction functor
$\ind_{\mf{l} + \uu}^\g: {}^{\lv}\OO_{\la} \rightarrow {}^{\g}\OO_{\la}$ is an equivalence
(of highest weight categories) with inverse $\res^\g_{\mf{l},\lambda}: {}^{\g}\OO_{\la}
\rightarrow {}^{\lv}\OO_{\la}$ given by $M \mapsto M^{\uu}$.
\end{proposition}

Denote by $\g_{[\la]}'=[\g_{[\la]}, \g_{[\la]}]$ the derived
subalgebra of $\g_{[\la]}$. Forgetting the action of the center of $\g_{[\la]}$ we see that
${}^{\lv}\OO_{\la}$ is equivalent to a block of ${}^{\g_{[\la]}'}\OO$ which is integral by construction.

\subsection{A sequence of reductions}

\begin{remark}\label{rem:interchange}
Let $\g=\gl(m)$ with simple system
$\{\epsilon_1-\epsilon_2,\ldots,\epsilon_{m-1}-\epsilon_m\}$. Let
$\la=\sum_{i=1}^m\la_i\epsilon_i$ be a weight and let
$\alpha=\epsilon_i-\epsilon_{i+1}$ be a simple root. We have
\begin{align*}
s_\alpha\cdot\la =
\la_1\epsilon_1+\ldots+\la_{i-1}\epsilon_{i-1}+(\la_{i+1}-1)\epsilon_{i}
+(\la_{i}+1)\epsilon_{i+1}+\la_{i+2}\epsilon_{i+2}+\ldots+
\la_{m}\epsilon_{m}.
\end{align*}
Hence $s_\alpha$ has the effect of interchanging the
$i$-th and the $i+1$-st coefficients of $\la$ modulo $\Z$.
\end{remark}

\begin{proposition} \label{prop:seq}
Let $\g=\gl(m)$ and $\la =\la_0 \in \h^*$. There exists
$\alpha_1, \ldots, \alpha_n \in \Pi$ such that for
$\la_{k}:=(s_{\alpha_k}s_{\alpha_{k-1}}\cdots s_{\alpha_1})
\cdot \la$, $1\le k \le n$, we have:
\begin{enumerate}[$($i$)$]
\item\label{prop:seq.1} $\langle \la_{k-1}, \alpha_{k}^\vee \rangle \not\in
\Z$ for $k=1,\ldots,n$;
\item\label{prop:seq.2} $\g_{[\la_n]}$ is a Levi subalgebra of $\g$.
\end{enumerate}
\end{proposition}

\begin{proof}

Write $\lambda=(\lambda_1,\lambda_2,\dots,\lambda_m)$ if $\la=\sum_{i=1}^m\la_i\epsilon_i$.
Define an equivalence relation $\approx$ on the multiset $\{\lambda_1,\lambda_2,\dots,\lambda_m\}$
via $\la_i\approx \la_j$ if and only if $\la_i-\la_j\in\Z$. Modulo integral shifts,
the claim of the proposition is equivalent
to the assertion that, using a sequence of elementary transpositions, where at each step we can
swap two neighboring elements belonging to different $\approx$-equivalence classes, the sequence
$(\lambda_1,\lambda_2,\dots,\lambda_m)$ can be transformed to a sequence
$(\mu_1,\mu_2,\dots,\mu_m)$ in which all equivalence classes are connected in the sense that
$\mu_i\approx\mu_j$ for some $i<j$ implies $\mu_i\approx\mu_s$ for all $i\leq s\leq j$.
The latter assertion is evident.
%
\end{proof}

\begin{corollary}\label{corequiv}
Let $\g=\gl(m)$ and $\la \in \h^*$. Then $\OO_{\la}$ is equivalent to
an integral block of $\OO$ for some semisimple Lie algebra
(of a perhaps smaller rank).
\end{corollary}

\begin{proof}
Set $\la =\la_0$, and let $\alpha_1, \ldots, \alpha_n \in \Pi$ be a sequence given by Proposition~\ref{prop:seq}.
By Theorem~\ref{th:twist}, applying to $\OO_{\la}$ first the functor $T_{\alpha_1}$, then
$T_{\alpha_2}$ and so on, gives an equivalence from $\OO_{\la}$ to $\OO_{\lambda_n}$.
As $\g_{[\lambda_n]}$ is a Levi subalgebra, the proof is completed applying
Proposition~\ref{prop:ind}.
\end{proof}

Corollary~\ref{corequiv} makes Soergel's equivalence much more explicit in the special
case of the Lie algebra $\gl(m)$. Soergel's approach uses the coinvariant algebra and works for all
semisimple Lie algebras, see \cite{Soe}.

\begin{example}
Let $\g=\gl(3)$ with simple roots $\Pi=\{\epsilon_1-\epsilon_2,\epsilon_2-\epsilon_3\}$. Let
$\la=(1,\frac{1}{2},1)$ and we have $\Phi_{\la}=\{\epsilon_1-\epsilon_3\}$ and $\g_{[\la]}\cong\mf{gl}(2)+\h$
(the latter is not a Levi subalgebra of $\g$). The functor
$T_{\alpha_2}$ induces an equivalence from $\mc{O}_{\lambda}$ to $\mc{O}_{s_{\alpha_2}\cdot\lambda}$.
Furthermore, $\Phi_{s_{\alpha_2}\cdot\la}=\{\epsilon_1-\epsilon_2\}$ and
$\g_{[s_{\alpha_2}\cdot\la]}\cong \mf{gl}(2)\oplus\gl(1)$ is now a
Levi subalgebra.
\end{example}

\subsection{Beyond type $A$}

Theorem~\ref{th:twist} and Proposition~\ref{prop:ind} remain valid in the case when
$\g$ is a Kac-Moody algebra. So for a weight $\la$ which satisfies
the Conditions \eqref{prop:seq.1} and \eqref{prop:seq.2} of Proposition~\ref{prop:seq}, we still
have an equivalence from $\OO_{\la}$ to some integral block (for a Kac-Moody Lie algebra of,
possibly, lower rank). However, for general Lie or Kac-Moody algebras,
Proposition~\ref{prop:seq} holds only for a proper subset of weights in $\h^*$. Here are some examples.

\begin{example}
Let $\{\epsilon_i|1\le i\le n\}$ be the standard orthonormal basis of $\R^n$,
$n\ge 4$. The root system $\Phi$ of type $D_n$ in $\R^n$ has roots $\{\pm\epsilon_i\pm\epsilon_j|i\not=j\}$.
The standard simple system for this root system is given below with its Dynkin diagram:
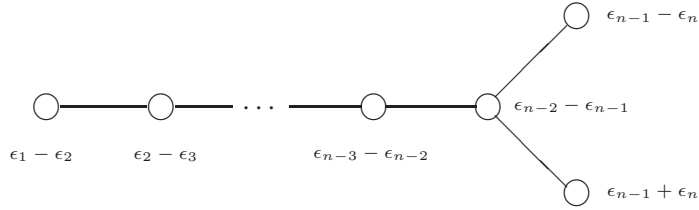
\begin{center}
\hskip -4cm \setlength{\unitlength}{0.25in}
\begin{picture}(24,4.4)
\put(8,2){\makebox(0,0)[c]{$\bigcirc$}}
\put(10.4,2){\makebox(0,0)[c]{$\bigcirc$}}
\put(14.85,2){\makebox(0,0)[c]{$\bigcirc$}}
\put(17.25,2){\makebox(0,0)[c]{$\bigcirc$}}
\put(19.1,3.9){\makebox(0,0)[c]{$\bigcirc$}}
\put(19.1,0.2){\makebox(0,0)[c]{$\bigcirc$}}
\put(8.3,2){\line(1,0){1.8}}
\put(10.7,2){\line(1,0){1.2}}
\put(13.1,2){\line(1,0){1.5}}
\put(15.1,2){\line(1,0){1.9}}
\put(17.4,2.2){\line(1,1){1.5}}
\put(17.4,1.8){\line(1,-1){1.5}}
\put(12.5,1.95){\makebox(0,0)[c]{$\cdots$}}
\put(7.9,1){\makebox(0,0)[c]{\tiny $\epsilon_1-\epsilon_2$}}
\put(10.5,1){\makebox(0,0)[c]{\tiny $\epsilon_2-\epsilon_3$}}
\put(14.8,1){\makebox(0,0)[c]{\tiny $\epsilon_{n-3}-\epsilon_{n-2}$}}
\put(19,2){\makebox(0,0)[c]{\tiny $\epsilon_{n-2}-\epsilon_{n-1}$}}
\put(20.7,0.2){\makebox(0,0)[c]{\tiny $\epsilon_{n-1}+\epsilon_{n}$}}
\put(20.7,3.9){\makebox(0,0)[c]{\tiny $\epsilon_{n-1}-\epsilon_{n}$}}
\end{picture}
\end{center}
Let
$\la=2\epsilon_1+\epsilon_2+\sum_{i=3}^n\frac{2n-2i+1}{2}\epsilon_i$.
Then $\Phi_{\la}$ has type $A_1\oplus A_1\oplus D_{n-2}$, where the simple
systems for $A_1$, $A_1$, and $D_{n-2}$ are given by
\begin{displaymath}
\{\epsilon_1-\epsilon_2\}, \{\epsilon_1+\epsilon_2\} \text{ and }
\{\epsilon_3-\epsilon_4,\epsilon_4-\epsilon_4,\ldots,\epsilon_{n-1}-\epsilon_n,\epsilon_{n-1}+\epsilon_n\},
\end{displaymath}
respectively. Thus, $\Phi_{\la}$ has rank $n$ and is a proper subset of $\Phi$. Therefore $\Phi_{\la}$
cannot coincide with the set of roots for any Levi subalgebra of a Lie algebra of type $D_n$.
This means that Proposition \ref{prop:seq} does not hold in type $D_n$, $n\ge 4$, in general.
\end{example}

\begin{example}\label{example:E8}
Let $\{\epsilon_i|1\le i\le 8\}$ be the standard orthonormal basis of $\R^8$.
The root system of type $E_8$ has roots
$\{\pm\epsilon_i\pm\epsilon_j|i\not=j\}\cup\{\frac{1}{2}\sum_{i=1}^8a_i\epsilon_i|a_i=\pm
1\text{ and }\prod a_i=1\}$. The standard simple system for this root system is given below with its
Dynkin diagram:
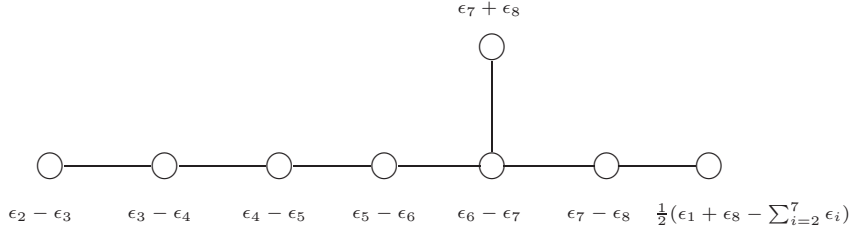
\begin{center}
\hskip -2cm \setlength{\unitlength}{0.25in}
\begin{picture}(24,5.5)
\put(5.6,2){\makebox(0,0)[c]{$\bigcirc$}}
\put(8,2){\makebox(0,0)[c]{$\bigcirc$}}
\put(10.4,2){\makebox(0,0)[c]{$\bigcirc$}}
\put(12.6,2){\makebox(0,0)[c]{$\bigcirc$}}
\put(14.85,2){\makebox(0,0)[c]{$\bigcirc$}}
\put(17.25,2){\makebox(0,0)[c]{$\bigcirc$}}
\put(19.4,2){\makebox(0,0)[c]{$\bigcirc$}}
\put(5.9,2){\line(1,0){1.8}}
\put(8.3,2){\line(1,0){1.8}}
\put(10.7,2){\line(1,0){1.6}}
\put(12.9,2){\line(1,0){1.7}}
\put(15.1,2){\line(1,0){1.9}}
\put(17.5,2){\line(1,0){1.6}}
\put(14.85,2.3){\line(0,1){1.9}}
\put(5.4,1){\makebox(0,0)[c]{\tiny $\epsilon_2-\epsilon_3$}}
\put(14.85,4.5){\makebox(0,0)[c]{$\bigcirc$}}
\put(7.9,1){\makebox(0,0)[c]{\tiny $\epsilon_3-\epsilon_4$}}
\put(10.3,1){\makebox(0,0)[c]{\tiny $\epsilon_4-\epsilon_5$}}
\put(12.6,1){\makebox(0,0)[c]{\tiny $\epsilon_5-\epsilon_6$}}
\put(14.8,1){\makebox(0,0)[c]{\tiny $\epsilon_6-\epsilon_7$}}
\put(17.1,1){\makebox(0,0)[c]{\tiny $\epsilon_7-\epsilon_8$}}
\put(14.8,5.3){\makebox(0,0)[c]{\tiny $\epsilon_7+\epsilon_8$}}
\put(20.3,1){\makebox(0,0)[c]{\tiny $\frac{1}{2}(\epsilon_1+\epsilon_8-\sum_{i=2}^7\epsilon_i)$}}
\end{picture}
\end{center}
Let
$\la=\frac{33}{2}\epsilon_1+2\epsilon_2+\epsilon_3+\frac{9}{2}\epsilon_4+\frac{7}{2}\epsilon_5+\frac{5}{2}\epsilon_6
+\frac{3}{2}\epsilon_7 + \frac{1}{2}\epsilon_8$. One checks that
$\Phi_{\la}$ has type $A_1\oplus E_7$, where the simple systems for $A_1$ and $E_7$ are $\{\epsilon_2-\epsilon_3\}$
and
\begin{displaymath}
\{\epsilon_4-\epsilon_5,
\epsilon_5-\epsilon_6, \epsilon_6-\epsilon_7, \epsilon_7-\epsilon_8,
\epsilon_7+\epsilon_8,
\frac{1}{2}(\epsilon_1+\epsilon_8-\sum_{i=2}^7\epsilon_i),
\epsilon_2+\epsilon_3\},
\end{displaymath}
respectively. Similarly to the previous example, this implies
that Proposition \ref{prop:seq} does not hold in type $E_8$ in general.
\end{example}

\begin{remark}
Similarly to Example \ref{example:E8} one shows that Proposition \ref{prop:seq} does not hold for
$E_6$ and $E_7$ in general. In particular, the proof of \cite[Proposition~A.4]{Ma} is valid
only in type $A$.
\end{remark}

\section{Equivalence of blocks for Lie superalgebras}\label{sec:superLA}

In this section, we generalize the above results
to Lie superalgebras of type $A$.

\subsection{Preliminaries}

Let $\C^{m|n}$ be the complex superspace of dimension $(m|n)$. The
general linear Lie superalgebra $\g=\gl(m|n)$ is the Lie superalgebra
of linear operators on $\C^{m|n}$. Let
$\{e_1,\ldots,e_m\}$ and $\{e_{m+1},\ldots,e_{m+n}\}$ be the
standard bases for the even subspace $\C^{m|0}$ and the odd subspace
$\C^{0|n}$, respectively. Their union is then a homogeneous basis
for $\C^{m|n}$ and we can use it to identify the space of $(m+n)\times (m+n)$ matrices
with $\gl(m+n)$. We let $e_{ij}$, $1\le i,j\le m+n$, denote the
$(i,j)$-th elementary matrix unit.
The Cartan subalgebra of diagonal matrices is denoted by
$\h=\h_{m|n}$, and it is spanned by $\{e_{ii}|1\le i\le m+n\}$. We
denote by $\{\delta_i|1\le i\le m+n\}$ the basis in
$\h^*=\h_{m|n}^*$ dual to $\{e_{ii}|1\le i\le m+n\}$. We let
$\Phi=\{\delta_i-\delta_j|1\le i\neq j\le m+n\}$ be the root system
and $W=\mf{S}_m\times\mf{S}_n$ be
the Weyl group (it acts by permuting the respective coordinates).

Set $[1,k]= \{1,2,\ldots, k\}$. Define the parity function $[1,m+n] \rightarrow \Z_2 =\{0,1\}$,
$i \mapsto \bar{i}$, where
\begin{align*}
\bar{i} =\begin{cases}
0, &\text{ if }1\le i\le m;\\
1, &\text{ if }m< i\le m+n.
\end{cases}
\end{align*}
The supertrace form on $\gl(m|n)$ induces a non-degenerate symmetric
bilinear form $(\cdot|\cdot)$ on $\h^*$ given by
\begin{align*}
(\delta_i|\delta_j)=\begin{cases}
(-1)^{\bar{i}}, &\text{ if }1\le i=j\le m+n; \\
0, &\text{ if }1\le i\neq j\le
m+n.
\end{cases}
\end{align*}
The subalgebra of upper triangular matrices with respect to this
standard basis is called the {\em standard Borel subalgebra} and is
denoted by $\mf{b}^{\text{st}}$. There exist
Borel subalgebras of $\gl(m|n)$ that are not $W$-conjugate to
$\mf{b}^{\text{st}}$. However, applying $W$-conjugation if necessary,
we may assume that a general Borel
subalgebra $\mf b$ satisfies $\mf b_{\bar 0}=\mf b_{\bar 0}^{\rm
st}$, and in such a case $\mf{b}$ is related to $\mf{b}^{\rm st}$ by a
sequence of {\em odd reflections} (see, e.g., \cite[Proposition
1.32]{CWbook}). The set of positive roots of $\mf b$ is denoted by
$\Phi^+_{\mf b}$.

The BGG category $\OO^{\mf b}$ for $\g$ with respect to $\mf b$ is
defined similarly as for Lie algebras. We let $\OO=\OO^{{\mf b}^{\rm
st}}$ denote the BGG category with respect to $\mf{b}^{\rm st}$.
As abstract categories, the categories $\OO^{\mf b}$ and $\OO$ coincide
(and hence are related by the identity functor). The identity functor
obviously sends simple objects to simple objects and projective objects to
projective objects. The category $\OO^{\mf b}$ is a highest weight category.
Standard objects for the highest weight structure are Verma modules
and the identity functor does not send standard objects to standard objects
in general. Another structural family of modules for a highest weight
category is formed by tilting modules.
For $\la\in\h^*$ denote by $M^{\mf b}(\la)$,
$T^{\mf b}(\la)$, and $L^{\mf b}(\la)$ the
Verma, tilting, and simple modules in $\OO^{\mf b}$ with highest
weight $\la$, respectively.
When $\mf b=\mf b^{\rm st}$, we will usually drop
the superscript.
All tilting modules are direct
summands of modules induced from tilting $\gl(m|n)_{\bar 0}$-modules, and this
implies that a $\mf{b}$-tilting module remains a $\mf b'$-tilting for
any other $\mf b'$ (which was first proved in \cite[Proposition~6.9]{CLW}). Therefore,
the identity functor sends tilting modules to tilting modules.

One can associate to a Borel subalgebra $\mf b$ with $\mf b_{\bar
0}=\mf b^{\textrm{st}}_{\bar 0}$ a sequence ${\bf
b}=(b_1,b_2,\ldots,b_{m+n})$ consisting of $m$ $0$s and $n$ $1$s,
called an $0^m1^n$-sequence as follows \cite[Section 6.1]{CLW}. Indeed it is well-known that such a Borel is completely determined by a total ordering of the basis $\{e_i|1\le i\le m+n\}$ such that the total ordering induced on basis elements of the same parity is the standard one. We then attach to $\mf b$ the sequence ${\bf b}$ by letting $b_j$ to be the parity of the $j$th basis element in this total ordering. For
example, the sequence
$\textbf{b}^\textrm{st}=(0^m,1^n)$ corresponds to $\mf{b}^\textrm{st}$.

Our next step is to define the {\em $\mf b$-standard partial order} on $\h^*$.
This is a straightforward generalization of the integral case dealt with in \cite{Br} and \cite[Section 2.3]{CLW}.
Let $\texttt{P}$ denote
the free abelian group with basis $\{\varepsilon_r\vert r\in\C\}$.
We define a partial order on $\texttt{P}$ by declaring $\nu\ge\mu$,
for $\nu,\mu\in \texttt{P}$, if $\nu-\mu$ is a non-negative integral
linear combination of $\varepsilon_r-\varepsilon_{r+1}$, $r \in \C$.

\begin{definition}
Fix a $0^m1^n$-sequence ${\bf b}=(b_1,\ldots, b_{m+n})$. For
$f: [1,m+n] \longrightarrow \C$ and $1\le j\le m+n$, we define
\begin{align*}
\text{wt}^j_{\bf b}(f)
 :=\sum_{j\le i}(-1)^{b_i}\varepsilon_{f(i)}\in \texttt{P},\quad
\text{wt}_{\bf b}(f):=\text{wt}^{1}_{\bf b}(f)\in \texttt{P}.
\end{align*}
Define the {\em standard partial order $\preceq_{\bf b}$ of type ${\bf
b}$} on the set of $\C$-valued functions on $[1,m+n]$ as
follows: $g\preceq_{\bf b}f$ if and only if $\text{wt}_{\bf
b}(g)=\text{wt}_{\bf b}(f)$ and $\text{wt}^j_{\bf
b}(g)\le\text{wt}^j_{\bf b}(f)$ for all $j$.
\end{definition}

Now the {\em $\mf b$-standard partial order} $\preceq_{\mf b}$ on $\h^*$
is defined as follows: Let ${\bf
b}$ be the $0^m1^n$-sequence associated with $\mf b$, and let
$\rho_{\mf b}$ be the Weyl vector normalized as in
\cite[(6.5)]{CLW}. We have a bijection between $\h^*$ and the
$\C$-valued functions on $[1,m+n]$ given by $\la\mapsto
f_\la$, where $f_\la(i)=(\la+\rho_{\mf b}|\delta_i)$. For  $\la,\mu\in\h^*$ we set $\la\preceq_{\mf b}\mu$ if $f_\la\preceq_{\bf b}f_\mu$.

\begin{remark}\label{rem:bruhat}
Let $\g=\gl(m|n)$ and $\la,\mu\in\h^*$ be such that $\la\preceq_{\mf b}\mu$. Suppose that
$\alpha$ is an even simple root for $\mf b$ such that
$\langle\la+\rho_{\mf b},\alpha^\vee\rangle\not\in\Z$. Then
$s_\alpha\cdot\la\preceq_{\mf b} s_\alpha\cdot\mu$.
\end{remark}

\subsection{A block decomposition}\label{sec:blockdec}

As in the case of Lie algebras, the category $\OO^{\mf b}$ has a
block decomposition, which we will describe below. First we note
that the definitions of $\Phi_\la$ and $\g_{[\la]}$ in Section
\ref{sec:setup} also make sense when $\g=\gl(m|n)$. The group
$W_\la$ is then the subgroup of the Weyl group $W$ generated by
$\{s_\alpha|\alpha\in\Phi_\la,\alpha\text{ even}\}$. For
$\la\in\h^*$, we let $\OO^{\mf b}_{\la}$ be the Serre subcategory of
$\OO^{\mf b}$ generated by simple objects of the form $L^{\mf b}(\mu)$
with $\mu=w\cdot(\la-\sum_{j}k_j\alpha_j)$, where $w\in W_{\la}$,
$\{\alpha_j\}$ is a set of mutually orthogonal odd isotropic roots
satisfying $\langle\la+\rho_{\mf b},\alpha_j^\vee\rangle=0$ (and
hence $\alpha_j\in\Phi_\la$), and $k_i\in\Z$. Denote the set of all
such $\mu$ above by $[\la]$ so that we have an equivalence relation
$\sim$ on $\h^*$ with equivalence classes $[\la]$. Note that
$\mu\preceq_{\mf b}\la$ implies that $\mu\in[\la]$.

\begin{proposition}\label{prop:block:decomp}
We have the following decomposition:
\begin{equation}\label{eq55}
{\mc O}^{\mf b}=\bigoplus_{\la\in\h^*/\sim}{\mc O}^{\mf b}_\la.
\end{equation}
\end{proposition}

\begin{proof}
Clearly ${\rm Hom}_{\mc O^{\mf b}}(L^{\mf
b}(\la),L^{\mf b}(\mu))\not=0$ implies that $\mu=\la$. Hence
we only need to check that ${\rm Ext}^1_{\mc O^{\mf b}}(L^{\mf
b}(\la),L^{\mf b}(\mu))\not=0$ implies that $\mu\in[\la]$.

First, for $\la=\sum_{i=1}^{m+n}\la_i\delta_i\in\h^*$ we define an
equivalence relation $\approx$ on the set $[1, m+n]$
similarly as before: for $i,j\in [1,m+n]$,
\begin{align}\label{equiv:rel:gl}
i\approx j,\quad\text{ if } (-1)^{\bar{i}}\la_i- (-1)^{\bar{j}}\la_j\in\Z.
\end{align}
Denote the equivalence classes by
$I^\la_1,I^\la_2,\ldots,I^\la_{\ell}$, where the numbering is
determined by the condition that $\min\{i|i\in I^\la_j\}<\min\{i|
i\in I^\la_{j+1}\}$, for $1\le j\le \ell-1$.

Suppose that ${\rm Ext}^1_{\mc O^{\mf b}}(L^{\mf b}(\la),L^{\mf
b}(\mu))\not=0$. Clearly, $\la$ and $\mu$ must have the same central
character. We have the following description of the central characters for $\g$, e.g.~\cite[Section~2.2.6]{CWbook}, which is a consequence of the description of the center of $U(\g)$ (see \cite{K, Sv}): Two weights $\la$ and
$\mu$ have the same central character if and only if there exist
$w\in W$ and a set of mutually orthogonal odd roots $\{\alpha_j\}_j$
such that $\mu=w\cdot(\la-\sum_{j}k_j\alpha_j)$ with $\alpha_j$
satisfying
$\langle\la+\rho_{\mf b},\alpha_j^\vee\rangle=0$ and $k_j\in\C$. We
observe that ${\rm wt}_{\bf b}(f_\la)={\rm wt}_{\bf b}(f_\mu)$,
where ${\bf b}$ is the $0^m1^n$-sequence corresponding to $\mf b$.
Looking at the supports of our modules, we obtain that
$\la-\mu\in\Z\Phi\subseteq\sum_{k=1}^{m+n}\Z\delta_k$, and thus
$I^\la_j= I^\mu_j$, for all $1\le j\le \ell$. From this it is
readily seen that one can find $w'\in W_\la$, $\beta_j\in\Phi_\la$,
and $l_j\in\Z$ such that $\langle\la+\rho_{\mf b},\beta_j^\vee\rangle=0$ and
\begin{align*}
\mu+\rho_{\mf b}=w'(\la+\rho_{\mf b}-\sum_{j}l_j\beta_j).
\end{align*}
The proposition is proved.
\end{proof}

\begin{remark}
Theorem \ref{thm:block:decomp} below shows that all summands
in \eqref{eq55} are indecomposable.
\end{remark}

\subsection{On integral weights}

A weight $\la \in \h^*$ is called {\em integral}  if $\langle\la, \alpha^\vee\rangle\in \Z$ for all $\alpha \in \Phi$.
Set
$$
{\bf 1}_{m|n} =\sum_{1\le i \le m+n} (-1)^{\bar{i}} \delta_i.
$$
Then for any integral weight
$\la \in \h^*$, we can write
$\la= {\la}' +a {\bf 1}_{m|n}$, where $a =a(\la) \in \C$ and
${\la}' =\sum_i {\la}_i' \delta_i \in \h^*$
satisfies that ${\la}_i' \in\Z$ for all $i$.
Denote by $V_a$ the one-dimensional $\glmn$-module
of weight $a {\bf 1}_{m|n}$. Tensoring with $V_a$ (and respectively, with $V_{-a}$ in the opposite direction) we
have the following.

\begin{lemma} \label{lem:1D}
There is an equivalence of blocks between $\OO_{\la}^{\mf b}$ and $\OO_{\la'}^{\mf b}$.
\end{lemma}

While $\la$ is integral in the sense of this paper, $\la'$ (but not $\la$ in general)
is an integral weight in the sense of \cite{Br, CLW}. This lemma assures us that we can work with
either version of integral weights.

\subsection{Parabolic induction functor}

We have the following generalization of Proposition \ref{prop:ind}.

\begin{proposition}\label{prop:ind:equiv}
Let $\mf{b}$ be a Borel subalgebra of $\g=\gl(m|n)$ and let $\la \in
\h^*$. Suppose that $\g_{[\la]}=\mf{l}$ is a Levi subalgebra of
$\g$ and let $\uu$ be the nilradical of $\mf{l}+\mf{b}$. Then the parabolic induction functor
$\ind_{\mf{l}+\mf{u}}^\g:{}^{\lv}\mc{O}_{\la}\rightarrow \mc{O}^{\mf
b}_\la$ is an equivalence of blocks, with inverse equivalence
$\res^\g_{\mf l}:\mc{O}^{\mf b}_\la\rightarrow {}^{\lv}\mc{O}_{\la}$
defined by $M\mapsto M^{\uu}$, where $M^{\uu}$ is the maximal trivial $\uu$-submodule of
$M$. Furthermore, this equivalence
maps Verma modules to Verma modules and hence is an equivalence of
highest weight categories.
\end{proposition}

\begin{proof}
Clearly, the parabolic induction functor sends Verma modules to Verma modules.
We shall now prove that it sends irreducible modules to irreducible
modules.

We have $\langle\la,\beta^\vee\rangle\not\in\Z$ for all roots  $\beta$ of $\mf u$ and so $\beta\not\in\Phi_\la$. Let $L^0(\la)$ be the
irreducible $\mf{l}$-module of highest weight $\la$ and consider the corresponding parabolically induced
module ${\ind_{\mf{l}+\mf{u}}^\g} L^0(\la)$. If
$\mu$ is the weight for a non-zero singular vector in a subquotient of the highest weight module
${\ind_{\mf{l}+\mf{u}}^\g} L^0(\la)$, then we have
$\mu\in[\la]$ by Proposition~
\ref{prop:block:decomp}. Since $\Phi_\la$ is invariant under $W_\la$, it follows that we can write $\mu=\la-\sum_{\alpha\in\Phi^+_{\mf b}\cap\Phi_\la}k_\alpha\alpha$, $k_\alpha\in\Z_+$. Hence any subquotient of ${\ind_{\mf{l}+\mf{u}}^\g} L^0(\la)$
intersects $L^0(\la)$, the latter being a simple $\mf{l}$-module.
This yields that ${\ind_{\mf{l}+\mf{u}}^\g} L^0(\la)$ is simple.

The module $E^0=\res^\g_{\mf l}L(\la)=L(\la)^{\mf u}$ is simple by a highest weight argument.
Indeed, assume that $E^0$ contains some simple submodule $L^0(\mu)$ different from
$L^0(\la)$. Then from the previous paragraph we have
${\ind_{\mf{l}+\mf{u}}^\g} L^0(\la)=L(\la)$ and
${\ind_{\mf{l}+\mf{u}}^\g} L^0(\mu)=L(\mu)$ and, by adjunction,
\begin{displaymath}
\mathrm{Hom}_{\mf{l}+\mf{u}}(L^0(\mu),\res^\g_{\mf l}{\ind_{\mf{l}+\mf{u}}^\g} L^0(\la))=
\mathrm{Hom}_{\g}({\ind_{\mf{l}+\mf{u}}^\g} L^0(\mu), {\ind_{\mf{l}+\mf{u}}^\g} L^0(\la))
=\mathrm{Hom}_{\g}(L(\mu),L(\la))\neq 0,
\end{displaymath}
which implies $\la=\mu$.

We can now conclude that $\res^\g_{\mf l}\ind_{\mf l+\mf u}^\g
L^0(\la)\cong L^0(\la)$ and $\ind_{\mf l+\mf u}^\g \res^\g_{\mf
l}L(\la)\cong L(\la)$. By induction on the length of the modules and
the Short Five Lemma we conclude that $\res^\g_{\mf l}$ and $\ind_{\mf l+\mf
u}^\g$ are mutually inverse equivalences of categories.
\end{proof}

\subsection{Equivalence via odd reflections}

Suppose that $\alpha$ is an isotropic odd simple root of $\mf b$. Let
$\mf{b}'$ be the Borel subalgebra obtained from $\mf b$ by applying
the odd reflection with respect to $\alpha$. The BGG categories $\mc O^{\bf b}$ and $\mc O^{\bf b'}$ are identical, although they are equipped with different highest weight structures, see e.g.~\cite[Section~6]{CLW}. The following lemma follows from the definition of the equivalence class
$[\la]$ in Section~\ref{sec:blockdec}.

\begin{lemma}\label{lem123}
If $\alpha$ is an isotropic simple root and
$\langle\la,\alpha^\vee\rangle\not\in\Z$, then for any $\mu \in
[\la]$ we have $\langle\mu,\alpha^\vee\rangle\not\in\Z$.
\end{lemma}

The following proposition describes the relation between standard (or simple)
modules for the different highest weight structures on $\mc O$
corresponding to $\mf b$ and $\mf b'$.

\begin{proposition}\cite[Lemma 1]{PS} \label{prop:oddref}
Assume that $\alpha$ is an isotropic simple root of $\mf b$ and
$\langle\la,\alpha^\vee\rangle\not\in\Z$. Then for every $\mu \in [\la]$
we have $M^{\mf b}(\mu)=M^{\mf b'}(\mu-\alpha)$ and
$L^{\mf b}(\mu)=L^{\mf b'}(\mu-\alpha)$.
\end{proposition}

\begin{proof}
Denote by $v_\mu^+$ the $\mf b$-highest weight vector in $M^{\mf
b}(\mu)$ or $L^{\mf b}(\mu)$, and by $f_\alpha$ a root
vector for $-\alpha$. Note that $e_\alpha f_\alpha v_\mu^+$ is a
nonzero multiple of $v_\mu^+$ thanks to
$\langle\mu,\alpha^\vee\rangle\not\in\Z$ (see Lemma~\ref{lem123}). Then it
is straightforward to verify that $f_\alpha v_\mu^+$ is a $\mf
b'$-highest weight vector in both $M^{\mf b}(\mu)$ and $L^{\mf
b}(\mu)$. Hence the irreducibility implies that
$L^{\mf b}(\mu)=L^{\mf b'}(\mu-\alpha)$. Moreover, we have a natural surjective
$\glmn$-homomorphism $M^{\mf b'}(\mu-\alpha) \rightarrow M^{\mf b}(\mu)$,
which must be an isomorphism by a character comparison.
\end{proof}

\subsection{Twisting functors for Lie superalgebras}

For an even simple root $\alpha$ of $\mathfrak{g}$ we can define the corresponding twisting endofunctor
$T_{\alpha}$ of $\mathcal{O}$ in exactly the same way as for Lie algebras, see
Section~\ref{sec2.2}. There is an obvious analogue of \cite[Lemma~3]{KM} for $T_{\alpha}$,
saying that this functor is essentially defined on the level of the
$\mathfrak{sl}(2)$-subalgebra of $\mathfrak{g}$. It follows that the Properties
\eqref{twist1}--\eqref{twist4} in Section~\ref{sec2.2} and \cite[Theorem~4.1]{HAS} transfer mutatis mutandis
to the superalgebra case. In particular, similarly to Theorem~\ref{th:twist}, we have the following.

\begin{proposition}\label{prop:twist}
Let $\la \in \h^*$ and $\alpha$ be an even simple root such that
$\langle \la, \alpha^\vee \rangle \not\in \Z$. Then
the functor $T_\alpha: \OO_\la \rightarrow \OO_{s_\alpha \cdot \la}$ is
an equivalence (as highest weight categories)
sending $M(w\cdot \la)$ to $M(s_\alpha w\cdot \la)$ and $L(w\cdot \la)$ to $L(s_\alpha w\cdot \la)$,
for each $w\in W$.
\end{proposition}

\subsection{Reduction to integral blocks}

Now consider a block $\mathcal O_\la$ with respect to the standard Borel $\mf b^{\text{st}}$ for an arbitrary weight $\la
\in \mathfrak h^*$. We can first apply a sequence of suitable
twisting functors (see Propositions~\ref{prop:seq} and \ref{prop:twist}) to $\mathcal O_\la$
and obtain an equivalent block $\mathcal O_{\widetilde{\la}}$ such
that the congruence classes modulo $\Z$ of $\widetilde{\la}_i$ for $1\le i \le
m$, and respectively for $m+1\le i \le
m+n$, are connected intervals, where $\widetilde{\la}
=\sum_{i=1}^{m+n} \widetilde{\la}_i \delta_i$. Next we can apply a
sequence of suitable odd reflections (see
Proposition~\ref{prop:oddref}) to $\mathcal O_{\widetilde{\la}}$ to
rewrite $\mathcal O_{\widetilde{\la}}$ in the form $\mathcal O_{\widetilde{\la}}^{\mf b'}$
for some non-standard Borel $\mf b'$ so that
$\g_{[\widetilde{\la}]}$ is a Levi subalgebra of $\g$ with respect
to the positive system associated to $\mf b'$. By
Proposition~\ref{prop:ind:equiv}, the block $\mathcal
O_{\widetilde{\la}}^{\mf b'}$ is equivalent to a block of integral
weights for $\g_{[\widetilde{\la}]}$ (or rather its derived
subalgebra). This proves the main result of this note.

\begin{theorem}\label{thmmain}
Every block $\mathcal O_\la$, $\la\in \mathfrak h^*$, is equivalent to an integral
block of $\mathcal O$ for some direct sum of general linear Lie superalgebras.
\end{theorem}

Note that, since the Lie superalgebra $\g_{[\widetilde{\la}]}$ is a direct sum of general linear
Lie superalgebras, the main theorem of \cite{CLW} for a general version of Brundan's conjecture
applies (with the help of the easy Lemma~\ref{lem:1D}).
Let us illustrate by example how the weights change following the
equivalences of blocks.

\begin{example}
Let $\g=\gl(5|3)$. Consider the block relative to a standard Borel
for a weight $\la$ whose corresponding ``$\rho$-shifted sequence"
$\{(\la+\rho, \delta_i)\}$ is given by $(8,2.1,6.7, 3,
2.7|5,4.1,9)$ (here and below we use a vertical line to indicate a
parity change, i.e., the appearance of an odd simple root). Via
twisting functors as in Theorem~\ref{th:twist}, we transform the
above weight to a weight $\widetilde{\la}$ with the $\rho$-shifted sequence
$(8,3,2.1,6.7,
2.7|5,9,4.1).$ Applying now several (odd reflection) functors as in
Proposition~\ref{prop:oddref}, we obtain a new weight $\widetilde{\widetilde{\la}}$ with the
$\rho$-shifted sequence
$(8,3|5,9|2.1|4.1|6.7, 2.7)$, i.e., the weight $\widetilde{\widetilde{\la}}$ is now regarded as a weight relative to the Borel subalgebra corresponding to the $0^51^3$-sequence $(00110100)$.  Hence, the Levi
subalgebra of $\g$ with respect to the new Borel
$\widetilde{\widetilde{\bf b}}$ here is $\gl(2|2)\times \gl(1|1)
\times \gl(2)$, and its ``integral weight" corresponds to a
$\rho$-shifted sequence $(8,3|5,9) \times (2.1|4.1) \times (6.7,
2.7).$
\end{example}

\subsection{Indecomposibility of blocks}

Recall from Proposition~\ref{prop:block:decomp} that
${\mc O}^{\mf b}=\oplus_{\la\in\h^*/\sim}{\mc O}^{\mf b}_\la$.

\begin{theorem}\label{thm:block:decomp}
The subcategory $\mc O^{\mf b}_\la$ is indecomposable for every $\la\in\h^*/\sim$.
\end{theorem}

\begin{proof}
In light of Theorem~\ref{thmmain} and its proof, it suffices to prove the theorem for $\la$ integral
and $\mf b$ the standard Borel subalgebra.

For $\eta, \nu \in \h^*$ write $\eta\smallsmile\nu$, if
$\text{Ext}^1_{\mc O}(L(\eta),L(\nu))\not=0$ or equivalently if $\text{Ext}^1_{\mc O}(L(\nu),L(\eta))\not=0$.
We want to prove that for any given $\mu\in[\la]$ there exists a sequence of weights $\eta_1,\eta_2,\ldots,\eta_p$
with $\eta_1=\la$ and $\eta_p=\mu$ and $\eta_i\smallsmile\eta_{i+1}$, for $1\le i\le p-1$.
We write $\la\circeq\mu$ if such a sequence exists. Note that if there exists an indecomposable
module $M$ of finite length such that $L(\la)$ and $L(\mu)$ are among its composition factors, then $\la\circeq\mu$.

First, suppose that $\mu\in[\la]$ with $s(\la+\rho)=\mu+\rho$, for some simple reflection $s\in W$,
and we have the usual argument as for semisimple Lie algebras. More precisely,
we have $\la+\rho$ and $\mu+\rho$ are identical, except that $(\la+\rho)_i=(\mu+\rho)_{i+1}$
and $(\la+\rho)_{i+1}=(\mu+\rho)_{i}$, for some $i\not=m$. We have either
$(\la+\rho)_i>(\la+\rho)_{i+1}$ or $(\la+\rho)_{i+1}>(\la+\rho)_{i}$. In the first case let $v_\la$ be a highest weight vector
in the Verma module $M(\la)$ and let $f_\alpha$ be the negative root vector corresponding to the simple root
$\alpha=\delta_i-\delta_{i+1}$. Then $f_\alpha^{\langle\la+\rho,\alpha^\vee\rangle}v_\la$ is a singular vector of
weight $\mu$ in $M(\la)$ and hence $\mu\circeq\la$. In the second case we have a singular vector of weight $\la$
in $M(\mu)$, and hence again $\mu\circeq\la$. It follows that if $w(\la+\rho)=\mu+\rho$, for some $w\in W$, then $\mu\circeq\la$.

Now suppose that $\alpha$ is an isotropic root such that $\langle\la+\rho,\alpha^\vee\rangle=0$, and $\la+\rho-\alpha=\mu+\rho$. If $\alpha$
is simple, then $f_\alpha v_\la$ is a singular vector of weight $\mu$ in $M(\la)$. Hence $\mu\circeq\la$. If $-\alpha$ is simple,
we reverse the role of $\la$ and $\mu$ and conclude again that $\mu\circeq\la$.
Now if $\alpha$ is not simple, then we can find an element $w\in W$ such that $w(\alpha)=\pm\beta$, with $\beta$ simple.
This implies that $\mu\circeq w(\mu+\rho)-\rho\circeq w(\la+\rho)-\rho\circeq\la$.

The general case is reduced to a combination of the above steps by the definition of the equivalence class $[\la]$ in Section~\ref{sec:blockdec}.
\end{proof}

\end{document}